\documentclass[12pt]{article}
\usepackage{amsmath, amsfonts, amsthm, amssymb, color, hyperref, extarrows, enumitem, nicefrac,blindtext, mathtools,}

\newtheorem{theorem}{Theorem}[section]
\newtheorem{definition}[theorem]{Definition}

\newtheorem{lem}[theorem]{Lemma}
\newtheorem{pro}[theorem]{Proposition}

\numberwithin{equation}{section}

\usepackage[hyperpageref]{backref}

\usepackage{cite}

\hypersetup{
	colorlinks   = true,
	citecolor    = magenta}
	
\begin{document}
\title{\vspace{-1cm} \bf Weighted Sobolev estimates of the truncated Beurling operator  \rm}
\author{Yifei Pan\ \  and\ \  Yuan Zhang}
\date{}

\maketitle

\begin{abstract}
 Given   a  bounded planar domain $D$ with  $W^{k+1, \infty}$ boundary, $ k\in \mathbb Z^+$,   and a weight  $\mu\in A_p, 1<p<\infty$, we show that   the corresponding truncated Beurling transform is a bounded operator  sending  $W^{k, p}(D, \mu)$ into itself. Weighted Sobolev estimates for other  Cauchy-type integrals are also obtained.  %As applications, we study weighted Sobolev regularity of the Cauchy-Riemann equations on product domains which further allows us to study Sobolev regularity on the Hartogs triangle.  

\end{abstract}

\renewcommand{\thefootnote}{\fnsymbol{footnote}}
\footnotetext{\hspace*{-7mm}
\begin{tabular}{@{}r@{}p{16.5cm}@{}}
& 2010 Mathematics Subject Classification. Primary 30E20; Secondary 42B37, 46E35.\\
& Key words and phrases.  Truncated Beurling transform, Cauchy-type integrals,  Muckenhoupt's class, weighted Sobolev regularity.
\end{tabular}}

\section{Introduction}
Let us first recall the well-known Beurling transform $B$ on  $f\in L^p(\mathbb C), 1<p<\infty$: 
$$  Bf : = p.v.\frac{-1}{2\pi 
            i}\int_{\mathbb C}\frac{f(\zeta)}{(\zeta-\cdot)^2} d\bar\zeta\wedge d\zeta\ \ \text{on}\ \ \mathbb C, $$
where  {\it p.v.} represents the principal value.  By the  classical Calderon-Zygumund theory \cite{CZ}, $B$ is a bounded operator sending $L^p(\mathbb C)$ (and $W^{k, p}(\mathbb C)$) into $L^p(\mathbb C)$ (and $W^{k, p}(\mathbb C)$, respectively). In this paper, we study the operator theory of  the associated  %whose boundary  consists of a finite family of nonintersecting rectifiable Jordan curves.  
truncated Beurling operator $H:  = \chi_D B(\chi_D\cdot)$ with respect to   a bounded  domain $D\subset \mathbb C$. Here $\chi_D$ is the characteristic function of $D$.  Namely,   for  $f\in  L^p(D), 1<p<\infty$,
\begin{equation}\label{op1H}
\begin{split}
  Hf &: = p.v.\frac{-1}{2\pi 
            i}\int_D\frac{f(\zeta)}{(\zeta-\cdot)^2} d\bar\zeta\wedge d\zeta \ \ \text{on}\ \  D.
\end{split}
\end{equation}
Unlike the boundedness of $B$ in $W^{k, p}(\mathbb C) $, the Sobolev regularity of $ H$ requires some necessary boundary regularity of the domain $D$.  For instance,  a polygon    in $\mathbb C$ has  $ W^{1, \infty}$ boundary. However, $\partial H(z)$ behaves like $ \frac{1}{|z-w_0|}$  near any of its vertices $w_0 $ on the boundary (see  \cite[pp. 146]{AIM}), which does not lie in $W^{1, p}$ any more if $p\ge 2 $. 
  %In this paper, we study the weighted Sobolev regularity of $H$.

Besides the Beurling operator,  there are two other Cauchy-type integral operators  that play a crucial role  in complex analysis and singular integral theory. For $1<p<\infty$, 
\begin{equation}\label{op1T}
\begin{split}
 Tf &: =\frac{-1}{2\pi i}\int_D \frac{f(\zeta)}{\zeta- \cdot}d\bar{\zeta}\wedge d\zeta \ \ \text{on}\ \  D,  \ \ \text{for}\ \ f\in  L^p(D);\\
 Sf &: =\frac{1}{2\pi i}\int_{bD}\frac{f(\zeta)}{\zeta- \cdot}d\zeta \ \ \text{on}\ \  D, \ \ \text{for}\ \ f\in  L^p(bD). 
\end{split}
\end{equation}
%There  has been fruitful study concerning  the operator theory of those Cauchy-type  integrals in \eqref{op1T}. For instance, 
The solid Cauchy  operator $T$ is known to provide a weak solution to the Cauchy-Riemann equation $\bar\partial u =f$ on $D$ and to improve the H\"older regularity by one order (see \cite{AIM, V});   the boundary Cauchy operator $S$ is  a  singular integral operator when restricting on $bD$,  which sends $L^p(b D)$ continuously  into itself according to a  result of Riesz. The truncated Beurling transform  $H$ is     related to  $T$ in terms of a distributional  derivative of $T$ from the point of view of complex analysis.% (see \eqref{tb}).

 %As a result of this,  $T$ is bounded from $L^{p}(D,)$ into $W^{1, p}(D)$, improving Sobolev regularity by order one precisely.  On the other hand, 

%Unlike the well-established H\"older theory, the Sobolev regularity of those   Cauchy-type integrals remains unclear until very recently. 
Based upon  harmonic analysis methods,  the sequential works of \cite{T, PT, P1}    assert  that the truncated Beurling transform $H$  maintains the (unweighted) $W^{k, p}$ regularity provided that  $p>2$ and the unit normal of $bD$ falls into some Besov space. % satisfies certain regularity conditions. 
% Since $H \partial f = \bar\partial f$ for $f\in W^{1,p}(\mathbb C)$, 
Ever since, such  Sobolev regularity of $H$  has  found substantial applications to  quasi-conformal mapping theory  through Beltrami-type equations \cite{CMO, P2}, and the Cauchy-Riemann equations (or the so-called $\bar\partial$ problem).  %As solutions to $\bar\partial$ equation on product domains \cite{NW, FP, CM} can be represented in terms of compositions of slice-wise Cauchy-type integrals, it can also be applied to study Sobolev regularity on such domains. Moreover, 
In particular, in the study of the $\bar\partial$ problem on some types of quotient domains that are realized as  proper holomorphic images of product domains,  a machinery   introduced in \cite{MM, CS} (see also \cite{YZ, Zhang2}) has shown that  the corresponding Sobolev regularity  can eventually be reduced to a  weighted Sobolev regularity for  those Cauchy-type integrals. See also a subsequent paper \cite{PZ} of the two authors for an immediate application to an optimal (unweighted) Sobolev regularity of $\bar\partial$ on product domains and the Hartogs triangle.

The goal of this paper, motivated by the results and the   call for weighted Sobolev theory to the Cauchy-type integrals,  %the boundedness of $H$ in a class of weighted $L^p$ spaces
 is to study the weighted Sobolev regularity of these  operators in \eqref{op1H}-\eqref{op1T}. Here our weight space is taken to be  the standard Muckenhoupt's $A_p$ class (see Section 2). Denote  by $\mathbb Z^+$   the set of positive integers and by $W^{k, p}(D, \mu)$ the weighted Sobolev space on $D$ with respect to a weight $\mu$, $k\in \mathbb Z^+\cup\{0\}, 1<p<\infty$.

\begin{theorem}\label{mainT}
Let $D\subset \mathbb C$ be a  bounded domain with $W^{k+1, \infty}$ boundary, $k\in \mathbb Z^+\cup\{0\} $. Assume $\mu\in A_p, 1<p<\infty$.  There exists a constant $C$ dependent only on $D, k$, $p$ and $\mu$, such that for all $f\in W^{k, p}(D, \mu)$,
    %  \begin{equation}\label{So}
     %         \|S f\|_{W^{k+1, p}(D, \mu)}\le C \|f\|_{W^{k+1, p}(D, \mu)};
      %\end{equation}
        \begin{equation}\label{Ho}
       \|H f\|_{W^{k, p}(D, \mu)}\le C \|f\|_{W^{k, p}(D, \mu)}
\end{equation}
and 
\begin{equation}\label{To}
      \|T f\|_{W^{k+1, p}(D, \mu)}\le C \|f\|_{W^{k, p}(D, \mu)}.
      \end{equation}
%If in addition $p>2$, and $\mu \equiv 1$, then the above holds as long as $bD\in  W^{k-\frac{1}{p}, p}$. 
\end{theorem}

 \medskip
 
The $k=0$ case in the   theorem above is the classical weighted Calderon-Zygmund theory, which is included   for the sake of completeness.  
 We stress  particularly that, in the case when  $\mu \equiv 1$ (the unweighted case), $p>2$,  and  the unit normal $\mathcal N$ of  $bD$ is in the Besov space $ B^{k-\frac{1}{p}}_{p, p}$ (equivalently, $bD\in B^{k+1-\frac{1}{p}}_{p, p} =W^{k+1-\frac{1}{p}, p}$), \eqref{Ho}  was  firstly proved by Prats in \cite{P1} using  technical harmonic analysis approach. %Note that by Sobolev embedding theorem, $ W^{k+1-\frac{1}{p},p}\subset W^{k,\infty}$ when $p>2$.
 In comparison to his result, our Theorem \ref{mainT}  recovers the Sobolev regularity with respect to any arbitrary $A_p$ weight over the full range of $p$, $1< p< \infty$,  by assuming  a slightly stronger boundary regularity $bD\in W^{k+1, \infty}$. Our proof  to Theorem \ref{mainT},    short and analytic, is mostly done through the line of complex analysis and partial differential equations. One of the main ingredients is  an induction formula on $S$, %(originally proved in \cite{YZ} for smooth domains), 
 together with a higher order Cauchy-Green formula. Making use of it,  we   reduce the weighted  Sobolev regularity of  $H$ and $T$ to the corresponding estimates in the  Lebesgue spaces, where the  Calderon-Zygmund theory applies.  
 
  It is worth pointing out that, under the same  assumption $\mu \equiv 1, p>2$ and $bD\in  W^{k+1-\frac{1}{p}, p}$ as in \cite{P1},  our proof  can actually  recover the above-mentioned result of Prats. For convenience of the reader, we have included the corresponding argument in our proof. The key observation is  that when $p>2$,  $W^{k, p}(D)$ forms  a multiplication algebra. %  under the  assumption $\mu \equiv 1, p>2$ and $bD\in  W^{k+1-\frac{1}{p}, p}$ as in \cite{P1}. %We also mention that  one can not replace the boundary assumption $W^{k+1, \infty}$ in Theorem \ref{mainT} by $W^{k, \infty} $. 

 %See also Remark \ref{re1} on further lowering the boundary regularity assumption in the case when $p>2$, which meanwhile recovers the result in \cite{P1}. 

 We also obtain the following weighted $W^{k, p}$ regularity for  $S$ on domains with $W^{k, \infty}$ boundaries (which is one order lower than that in Theorem \ref{mainT}).   %Note that if $p>2$, the boundary condition $ B^{k+1-\frac{1}{p}}_{p, p}$ in \cite{P1} is a proper subset of  $ W^{k,\infty} $  by the Sobolev embedding theorem. 
 Note that  the $k=0$ case is excluded since $S$ is not well-defined there. 
 
 \begin{theorem}\label{mainS}
 Let $D\subset \mathbb C$ be a  bounded domain with $W^{k, \infty}$ boundary, $k\in \mathbb Z^+$. Assume $\mu\in A_p, 1<p<\infty$.  There exists a constant $C$ dependent only on $D, k$, $p$ and $\mu$, such that for all $f\in W^{k, p}(D, \mu)$,
      \begin{equation}\label{So}
              \|S f\|_{W^{k, p}(D, \mu)}\le C \|f\|_{W^{k, p}(D, \mu)}.
      \end{equation}
If in particular $\mu \equiv 1$ and $p>2$, then the above holds as long as $D$ has $W^{1, \infty}\cap W^{k-\frac{1}{p}, p}$ boundary. 
      \end{theorem}

\section{Notations and preliminaries}

 Denote by $dV$ the Lebesgue integral element on $\mathbb C$, and by  $|S|$  the Lebesgue measure of a subset  $S$ in $\mathbb C $. Our weights   under consideration are in the standard Muckenhoupt space $A_p$ as follows. %For  $z=(z_1, \cdots, z_n)\in \mathbb C^n$, let $\hat z_j =(z_1, \cdots, z_{j-1}, z_{j+1}, \cdots, z_n)\in \mathbb C^{n-1}$, where the $j$-th component of $z$ is skipped. 

\begin{definition}\label{aps}
Given $1<p<\infty$, a weight $\mu: \mathbb C \rightarrow [0, \infty)$ is said to be in $ A_p$ if its $  A_p  $   constant
$$ A_p (\mu): = \sup \left(\frac{1}{|B |}\int_{B }\mu(z)dV_{z }\right)\left(\frac{1}{|B |}\int_{B } \mu(z)^{\frac{1}{1-p}}dV_{z}\right)^{p-1}<\infty, $$
 where the supremum is taken over    all discs $B \subset \mathbb C$. %The infimum of such $C$'s is said to be the $A_p $ constant of $\mu$.
\end{definition}

  See \cite[Chapter V]{Stein} for an introduction of the $A_p$ class. Clearly, $A_q\subset A_p$ if $1< q<p<\infty$. $A_p$ spaces also satisfy an open-end property:  if $\mu\in A_p$ for some $1<p<\infty$, then  $\mu\in A_{\tilde p} $ for some ${\tilde p}<p$. As a direct consequence of this property and H\"older inequality, if $\mu\in A_p, 1<p<\infty$, there exists some $q>1$ such that $ L^p(D, \mu)\subset L^q(D)$ for a bounded domain $D$.

\medskip

   Given a non-negative weight $\mu$ and $ 1< p<\infty $, the weighted function space $W^{k, p}(D, \mu)$ is the set of  functions $f$ on $D$ whose weak derivatives up to order $k$ exist, and the  $W^{k, p}$ norm $$ \|f\|_{W^{k,p}(D, \mu)}: = \left(\sum_{j=0}^k\int_D |\nabla^jf(z)|^p\mu(z)dV\right)^\frac{1}{p}<\infty. $$
Here $\nabla^j f$ represents all $j$-th order weak derivatives of $f$.  When $\mu\equiv 1$, $W^{k, p}(D, \mu)$ is reduced to the (unweighted) $W^{k,p}(D)$ space. From now on, we shall say $a\lesssim b$ if  $a\le Cb$ for a constant $C>0$ dependent only possibly on $D, k, p$ and the $  A_p  $   constant of $\mu$. %We say  $a\approx b$ if and only if $a\lesssim b$ and $b\lesssim a$ at the same time. 

 \medskip

\section{Weighted Sobolev estimates  }

Let $D$ be a bounded domain in $\mathbb C$ with $W^{1, \infty}$  (equivalently, Lipschitz) boundary.  For $f\in L^p(D)$, it is known that the solid  Cauchy integral  $T$ is a solution operator to the $\bar\partial$ equation on $D$: $$\bar\partial T f =f\ \ \ \text{on}\ \ \ D$$ in the sense of distributions.  Moreover, for a.e. $z\in D$, %$T$ is  related to the truncated Beurling operator $H$ in terms of the following manner.
  \begin{equation}\label{tb}\begin{split}
     \partial T f(z)= Hf(z) \left(= \lim_{\epsilon\rightarrow 0} \frac{-1}{2\pi i}\int_{D\setminus B_\epsilon(z) }\frac{f(\zeta)}{(\zeta-z)^2} d\bar\zeta\wedge d\zeta\right),
 \end{split}
 \end{equation}
where $B_\epsilon(z)$ is the disc centered at $z\in D$ with radius $\epsilon$.  See for instance \cite{V,AIM}.    %$H$ is also known as the truncated Beurling transform on $D$. 
According to the weighted Calderon-Zygmund theory, if $\mu\in A_p, 1<p<\infty$, then
\begin{equation}\label{H}
    \|Hf\|_{L^{p}(D, \mu)}\lesssim \|f\|_{L^{p}(D, \mu)}.
\end{equation} 
The following proposition follows immediately from  the above (in)equalities and \cite[Proposition 3.1]{YZ} about a weighted estimate on $T$. 

  \begin{pro}\cite{Stein, V}\label{V}
 Let  $D$ be a bounded domain in $\mathbb C$ with $W^{1, \infty}$ boundary and  $\mu\in A_p,  1<p<\infty$. If $f\in L^p(D, \mu)$, then $Tf\in W^{1, p}(D, \mu)$ with
 \begin{equation*}
       \|Tf\|_{W^{1, p}(D, \mu)}\lesssim  \|f\|_{L^p(D, \mu)}.
      \end{equation*}
\end{pro}
\medskip

We first estimate the boundary Cauchy integral $S$. Recall that   $S$ is related to $T$ in terms of the following Cauchy-Green formula: for all $f\in W^{1, q}(D), q>1$.
\begin{equation}\label{CG}
    Sf = f- T\bar\partial f \ \ \ \text{on}\ \ \ D
\end{equation}
in the sense of distributions. Since  $W^{1,p}(D, \mu)\subset W^{1, q}(D)$  for some $q>1$, by the trace theorem $W^{1,p}(D, \mu)$  is continuously embedded in  $L^q(bD)$. So $Sf$  well defined on $W^{1,p}(D, \mu) $.  Moreover, we have the following estimate for $\partial S$.
\medskip

\begin{pro}\label{S1}
Let $D$ be a bounded domain with $W^{1, \infty}$ boundary and $\mu\in A_p, 1<p<\infty$. Let $f\in W^{1, p}(D, \mu)$, $1<p<\infty$. Then \begin{equation}\label{S}
  \partial Sf  = H\bar\partial f + \partial  f.
\end{equation} in the sense of distributions.  Consequently, 
\begin{equation}\label{Si}
    \|Sf\|_{W^{1, p}(D, \mu)}\lesssim \|f\|_{W^{1, p}(D, \mu)}.
\end{equation}
\end{pro}

\begin{proof}
Formula \eqref{S} follows from \eqref{CG} by taking $\partial$ directly in the sense of distributions. By \eqref{H}, $\partial S f\in L^p(D, \mu)$. To prove the estimate \eqref{Si}, we  first use the  Cauchy-Green formula \eqref{CG} to obtain
\begin{equation}\label{33}
    \|Sf\|_{L^p(D, \mu)}\lesssim \|f\|_{L^p(D, \mu)}+\|T\bar \partial f\|_{L^p(D, \mu)}\lesssim  \|f\|_{W^{1,p}(D, \mu)}.  
\end{equation}
%By holomorphic Hardy theory and a result of Riesz,  we have 
%\begin{equation}\label{33}
 %   \|Sf\|_{L^q(D)}\lesssim \|Sf\|_{L^q(bD)}\lesssim \|f\|_{L^q(bD)}\lesssim \|f\|_{W^{1,p}(D, \mu)}.  
%\end{equation}
 The rest of the proof of \eqref{Si} is a direct consequence of \eqref{33}, \eqref{S}, \eqref{H} and the holomorphy of $S$ on $D$.

\end{proof}

On the higher order Sobolev spaces, $S$ satisfies an induction formula below,  % as long as %The general weighted $W^{k, p}$ estimate of $S $ is a consequence of the following induction formula.   the unit normal $\mathcal N$ of $bD$  can be extended as an element  in $W^{k-1, \infty}(D)$, which will be sufficed if $D$ has $W^{k, \infty}$ boundary. 
similar to  \cite[Lemma 3.5]{YZ}. We include the proof here  for the reader's convenience, and meanwhile in order to chase down precisely how the regularity of $bD$ affects the regularity of $S$.  Let
\begin{equation*} 
    \partial S f =  \frac{1}{2\pi i}\int_{bD}\frac{f(\zeta)}{(\zeta- \cdot)^2}d\zeta=: \tilde S f\ \ \text{on}\ \ D.\end{equation*}
In particular, Proposition \ref{S1} states  that if $D\in W^{1,\infty}, \mu\in A_p, 1<p<\infty$, then 
\begin{equation}\label{ts}
    \|\tilde Sf\|_{L^p(D, \mu)}\lesssim \|f\|_{W^{1, p}(D, \mu)}.
\end{equation}
\medskip

\begin{lem}\label{ind}Let $D$ be a bounded domain in $\mathbb C$ with $W^{k, \infty}$ boundary, $ k\in \mathbb Z^+ $, and $\mu\in A_p,  1<p<\infty$.  For any $f\in W^{k, p}(D, \mu)$, we have
$$  \partial^k S f =    \tilde S \tilde f \ \ \text{on}\ \ D$$ for some $\tilde f\in W^{1, p}(D, \mu)$  such that \begin{equation}\label{sb}
    \|\tilde f\|_{W^{1, p}(D, \mu) }\lesssim \|f\|_{W^{k, p}(D, \mu) }.   
\end{equation}
Furthermore, if  $\mu \equiv 1$ and $p>2$, then the above still holds as long as $D$ is of $W^{1, \infty}\cap W^{k-\frac{1}{p}, p}$ boundary. 
%$$\frac{1}{2\pi i}\int_{\partial D} \frac{f(\zeta)d\bar\zeta}{(\zeta - z)^2} = \frac{1}{2\pi i}\int _{\partial D} \frac{\tilde f(\zeta)d\bar\zeta}{\zeta - z} =   K \tilde f$$ for some $\tilde f \in W^{1, p}(D)$ such that $\|\tilde f\|_{W^{1, p}(D)}\le C\|f\|_{W^{2, p}(D)}$. Hence $\|\frac{1}{2\pi i}\int_{\partial D} \frac{f(\zeta)d\bar\zeta}{(\zeta - z)^2}\|_{L^{ p}(D)}\le C\|f\|_{W^{2, p}(D)}.$
\end{lem}

\begin{proof}
% Since $C^\infty(\bar D)$ is dense in $W^{k, p}(D)$, we shall only prove the lemma for $f\in C^\infty(D)$.  
 The $k=1$ case is by the definition of $\tilde S$, with $\tilde f = f$.  When $k =  2$, let $\zeta$ be a parameterization of $bD$ in terms of the arclength $s\in (0, s_0)$. For  each fixed $z\in D$, since $\frac{1}{\cdot-z}\in C^\infty(bD) $ and $f\in W^{2, p}(D, \mu)\subset W^{2,1}(D)\subset W^{1, 1}(bD)$ by the trace theorem,   we can use the Divergence theorem on $bD$ to compute $\partial^2 Sf$ directly as follows.  
\begin{equation*}
    \begin{split}
 \partial^2   S f(z) =\partial \tilde Sf(z)   &=  \frac{1}{2\pi i} \int_{0}^{s_0} \partial_z\left(\frac{1}{(\zeta(s) - z)^2} \right) f(\zeta(s)) \zeta'(s)ds\\
 &= -\frac{1}{2\pi i} \int_{0}^{s_0} \partial_{s}\left(\frac{1}{(\zeta(s) - z)^{ 2}} \right) f(\zeta(s))ds\\
 &= \frac{1}{2\pi i} \int_{0}^{s_0} \frac{\partial_{s}\left(f(\zeta(s))\right)}{( \zeta(s) - z )^2}  ds\\
 &= \frac{1}{2\pi i} \int_0^{s_0} \frac{ \partial_\zeta f(\zeta(s))\zeta'(s) +\partial_{\bar\zeta} f(\zeta(s))\bar\zeta'(s) }{(\zeta(s) - z)^2}  ds\\
 &= \frac{1}{2\pi i} \int_{b D} \frac{\partial  f(\zeta) +\bar\zeta'^2(s)\bar\partial f(\zeta) }{(\zeta - z)^2}  d \zeta =   \tilde S\tilde f(z),
    \end{split}
\end{equation*}
where \begin{equation}\label{tf}
    \tilde f  =  \partial f +\bar\zeta'^2\bar\partial f\ \ \ \text{on}\ \ \ bD.
\end{equation}  
In the above we have used the fact that $\bar\zeta' = \frac{1}{\zeta'}$ on $bD$.

If  $bD$ has  $W^{2, \infty} $ boundary, then $\bar\zeta' \in W^{1,\infty}(bD)$. Namely, $\bar\zeta'$ is Lipschitz continuous on $bD$.  Making use of Kirszbraun theorem \cite{Ki}, %(or a similar idea as in \cite[Lemma 6.38]{GT} by first  extending $\bar\zeta'$ to be in $ W^{1,\infty}$ near a neighborhood of $bD$, then extending to  $W^{1, \infty}(D)$ with an application of the Sobolev extension theorem), 
 we can extend $\bar\zeta'$ to $D$, still denoted as $\bar\zeta'$, such that  $\bar\zeta'\in  W^{1,\infty}(D)$. Thus by \eqref{tf}, $\tilde f$ extends as an element in $W^{1, p}(D, \mu)$, with
$$ \|\tilde f\|_{W^{1, p}(D, \mu) }\lesssim \|\nabla f\|_{W^{1, p}(D, \mu) }\le \|f\|_{W^{2, p}(D, \mu)}.$$
If $\mu \equiv 1$, $p>2$ and $bD\in  W^{2-\frac{1}{p}, p}$ instead, then      $\zeta' \in  W^{1-\frac{1}{p}, p}(bD)$ (note that the unit normal $\mathcal N$ is $\pm i\zeta'$). By the trace theorem  for Sobolev–Slobodeckij spaces \cite{Ne, Ma}, $\zeta' $ %= -i\mathcal N$ 
can be extended as an element, still denoted by $\zeta'$, lying in $W^{1, p}(D)$. On the  other hand, by the Sobolev embedding theorem,   $\bar\zeta',  \nabla   f \in W^{1, p}(D )\subset L^{\infty}(D) $ with $\|  \bar\zeta'\|_{L^\infty(D)}\lesssim  \|    \bar\zeta' \|_{W^{1, p}(D )}\lesssim 1$ and $\|\nabla f\|_{L^\infty(D)}\lesssim  \|    f \|_{W^{2, p}(D )}$. Then 
$$\|\bar\zeta'^2  \|_{W^{1, p}(D ) }\lesssim  \|\bar\zeta'   \|_{W^{1, p}(D ) }\|\bar\zeta'\|_{L^\infty(D)}\lesssim 1. $$
Consequently,
\begin{equation*}
     \|\bar\zeta'^2\bar\partial f \|_{W^{1, p}(D ) }  \lesssim    \|\bar\zeta'^2\bar\partial f\|_{L^\infty(D)} + \|\bar\zeta'^2  \|_{W^{1, p}(D ) }\|\bar\partial f\|_{L^\infty(D)} + \| \bar\partial f \|_{W^{1, p}(D ) }\|\bar\zeta'^2\|_{L^\infty(D)}    \lesssim   \|f\|_{W^{2, p}(D )},
                \end{equation*}
from which  and \eqref{tf} we have $\tilde f\in W^{1, p}(D )$ and \eqref{sb} follows.   This completes the proof of the lemma  when $k=2$.  %and  $c_{01}$  be  in $ W^{1, \infty}(D)$ such that they equal to $1$ and  $\bar\zeta' (\in W^{1,\infty}(bD))$, respectively, when restricted on the parametrized boundary.  
The remaining part of the lemma follows by a standard induction on $k$.

\end{proof}

\begin{proof}[Proof of Theorem \ref{mainS}: ] Let $f\in W^{k, p}(D, \mu)  $.  The $k = 1$ case is due to Proposition \ref{S1}. When $k\ge 2$, since $Sf$ is holomorphic on $D$, we only need to estimate $\|\partial^k Sf\|_{L^p(D, \mu)}$. According to Lemma \ref{ind}, $  \partial^k S f =    \tilde S \tilde f$ for some $\tilde f\in W^{1, p}(D, \mu)$ such that 
$\|\tilde f\|_{W^{1, p}(D, \mu)}\lesssim \|f\|_{W^{k, p}(D, \mu)}. $
Then \eqref{ts} gives 
$$ \|\partial^k Sf\|_{L^p(D, \mu)} = \|\tilde S\tilde f\|_{L^{ p }(D, \mu)}\lesssim \|\tilde f\|_{W^{1, p }(D, \mu)}\lesssim  \|f\|_{W^{k, p}(D, \mu)}. $$ \eqref{So} is thus proved.  

\end{proof}

\medskip

%To prove Theorem \ref{mainT}, we first need some preparation results on  the operator $S$. 

\begin{proof}[Proof of Theorem \ref{mainT}: ] 
Since \eqref{Ho} is a consequence of \eqref{To} and \eqref{tb}, we will only estimate   $T$ on $W^{k, p}(D, \mu), k\ge 1$. Let \begin{equation*}
\begin{split}
    & L f: = \frac{1}{2\pi i}\int_{\partial D} \frac{f(\zeta)}{\zeta -\cdot}d\bar\zeta \ \ \  \text{on}\ \ \  D. 
\end{split}
     \end{equation*}
Similar to $S$, $Lf$ is well defined for $f\in W^{1,p}(D, \mu) $ and    is holomorphic in $D$.  
 $L$ is associated with $T$ in terms of the following   formula (see,  \cite{V} or \cite[Theorem 3.3]{YZ}):      % for $f\in W^{1, p}(D), 1<p<\infty$ 
\begin{equation*}
\begin{split}
       &\partial Tf  = T \partial f-Lf\ \ \ \text{on}\ \ \ D
        \end{split}
\end{equation*}
 in the sense of distributions.  Consequently,   for $j\in \mathbb Z^+, j\le k,$
 $$\partial^{j+1} Tf  = \partial^{j}T \partial f-\partial^{j} Lf\ \ \ \text{on}\ \ \ D  $$
 in the sense of distributions. Thus to prove \eqref{To},  % In view of formula \eqref{tb}, 
 we only need to show
 \begin{equation}\label{Lo}
      \|  Lf\|_{W^{k, p}(D, \mu)}\lesssim\|f\|_{W^{k, p}(D, \mu)}.
 \end{equation}

 As in the proof of Lemma \ref{ind}, let $\left.\zeta(s)\right|_{s\in (0, s_0)}$ be a parametrization of $bD$ in terms of the arclength parameter $s$.  Then for $z\in D$, % $Lf$ is related to $Sf$ as follows via the  parametrization $\zeta(s)$.
\begin{equation}\label{LS}
    \begin{split}
         Lf(z) = \frac{1}{2\pi i}\int_0^{s_0}\frac{f\left(\zeta(s)\right)}{\zeta(s)- z}\bar \zeta'(s)ds = \frac{1}{2\pi i}\int_0^{s_0}\frac{f(\zeta(s))(\bar\zeta')^2(s)}{\zeta(s)- z} \zeta'(s)ds =   S\left(\bar\zeta'^2f\right)(z).
    \end{split}
\end{equation}
If $bD\in W^{k+1, \infty}$, then $\zeta'\in W^{k, \infty}(bD)$ and % and the fact that $bD\in W^{k+1, \infty}$, we have $\zeta'\in  W^{k, \infty}(D)$ and 
$$     \|\bar \zeta'^2 f\|_{W^{k, p}(D, \mu) }\le  \|f\|_{W^{k, p}(D, \mu)}\| (\bar \zeta')^2\|_{W^{k, \infty}(bD)} \le\|f\|_{W^{k, p}(D, \mu)}\| \bar \zeta'\|^2_{W^{k, \infty}(bD)}  \lesssim \|f\|_{W^{k, p}(D, \mu)}.$$
 In the case when  $\mu \equiv 1$, $p>2$ and $bD\in  W^{k+1-\frac{1}{p}, p}$ (as assumed in \cite{P1}), we argue similarly as in the last paragraph of the proof to Lemma \ref{ind}. Indeed,  since $\bar \zeta'\in W^{k-\frac{1}{p}, p}(bD)$, % by the fractional Sobolev embedding theorem % (note that the unit normal in \cite{P1} is $i\zeta'$). Then 
 by the trace theorem  %\cite{Ne, Ma}, 
 $ \bar \zeta' $ %= -i\mathcal N$ 
can be extended as an element  in $W^{k, p}(D)$. Recall that $W^{k, p}(D )$ forms  a multiplication algebra when $p>2$. See  \cite{RS, P2} etc. Or, directly  by the Sobolev embedding theorem, we   have $  \bar \zeta',     f \in W^{k, p}(D )\subset W^{k-1, \infty}(D) $. Then $ \bar \zeta'^2 f   \in W^{k, p}(D ) $ with % with $\|  f\|_{W^{k-1, \infty}(D)}\lesssim  \|    f \|_{W^{k, p}(D )}$. 
\begin{equation*}
\begin{split}
                 \| \bar \zeta'^2 f  \|_{W^{k, p}(D ) }  \lesssim  &   \|\bar\zeta'^2  \|_{W^{k, p}(D ) }\|  f\|_{W^{k-1,\infty}(D)} + \|  f \|_{W^{k, p}(D ) }\|\bar\zeta'^2\|_{W^{k-1,\infty}(D)}  \\
                 \lesssim & \|(\bar\zeta')^2\|_{W^{k, p}(D )  } \|f\|_{W^{k, p}(D)}\lesssim \|  \bar\zeta'\|_{W^{k, p}(D)} \|\bar\zeta'\|_{W^{k-1, \infty}(D )  } \|f\|_{W^{k, p}(D)}\\
        \lesssim & \|\bar\zeta'\|^2_{W^{k, p}(D )  } \|f\|_{W^{k, p}(D )}\lesssim \|f\|_{W^{k, p}(D )}.
\end{split}
    \end{equation*}
%and the last inequality holds true by using a similar argument as in the last paragraph of the proof to Lemma \ref{ind}.  
In both cases, we combine  \eqref{So}  with \eqref{LS}   to obtain \eqref{Lo}. The proof is complete.

\end{proof}
\medskip

%\begin{remark}\label{re1}
%As  shown in the proof of Lemma \ref{ind} and Theorem \ref{mainT}, the boundary regularity $W^{k, \infty}$ is only used  to extend the boundary parameterization $\zeta'$ inside $D$. In the case when $p>2$ and $\mu \equiv 1$, we can further lower the boundary regularity  by assuming  the unit normal $\mathcal N \in  B_{p,p}^{k-\frac{1}{p}} $ (or equivalently, $ \mathcal N \in W^{k-\frac{1}{p}, p}$) instead  as in \cite{P1}. Indeed, if so  then by the trace theorem  for Sobolev–Slobodeckij spaces \cite{Ne, Ma}, $\zeta' = -i\mathcal N$ can be extended as an element, still denoted by $\zeta'$, lying in $W^{k, p}(D)$. Since $p>2$, by the Sobolev embedding theorem, $W^{k, p}(D )\subset W^{k-1, \infty}(D) $. Thus $W^{k, p}(D )$ forms  a multiplication algebra such that $$ \|h\cdot g\|_{W^{k, p}(D )}\lesssim \|h\|_{W^{k, p}(D )  }\|g\|_{ W^{k-1, \infty}(D )} + \|g\|_{W^{k, p}(D )  }\|g\|_{ W^{k-1, \infty}(D )} \lesssim \|h\|_{W^{k, p}(D ) } \|g\|_{W^{k, p}(D ) }$$ for $h, g\in W^{k, p}(D)  $. See also \cite{RS, P2} etc.  Consequently, $W^{1, p}(D)$ norm of $\tilde f$ in Lemma \ref{ind} is  controlled by the $W^{k, p}(D )$ norm of $f$. Namely, our short proof  recovers the unweighted result in \cite{P1}.
%\end{remark}
\medskip

\bibliographystyle{alphaspecial}

\fontsize{11}{11}\selectfont
\vspace{0.7cm}

\noindent pan@pfw.edu,

\vspace{0.2 cm}

\noindent Department of Mathematical Sciences, Purdue University Fort Wayne, Fort Wayne, IN 46805-1499, USA.\\

\vspace{0.5cm}

\noindent zhangyu@pfw.edu,

\vspace{0.2 cm}

\noindent Department of Mathematical Sciences, Purdue University Fort Wayne, Fort Wayne, IN 46805-1499, USA.\\
\end{document}